\documentclass[12pt,reqno]{amsart}
\usepackage{amsmath,amsfonts,amssymb,amsthm,amsbsy,latexsym,cite,amscd}
\usepackage[colorlinks,breaklinks,linkcolor=blue,citecolor=blue,menucolor=blue]{hyperref}

\multlinegap=\parindent

\begin{document}

\newtheorem{etheorem}{Theorem}
\newtheorem{eproposition}{Proposition}
\newtheorem*{ecorollary}{Corollary}

\renewcommand{\thefootnote}{}

\title[On~the~conjugacy separability of~Baumslag--Solitar groups]{On~the~conjugacy separability of~ordinary and~generalized Baumslag--Solitar groups}

\author{E.~V.~Sokolov}
\address{Ivanovo State University, Russia}
\email{ev-sokolov@yandex.ru}

\begin{abstract}
Let $\mathcal{C}$ be~a~class of~groups. A~group~$X$ is~said to~be~residually a~$\mathcal{C}$\nobreakdash-group (conjugacy $\mathcal{C}$\nobreakdash-sepa\-ra\-ble) if, for~any elements $x,y \in X$ that are~not~equal (not~conjugate in~$X$), there exists a~homomorphism~$\sigma$ of~$X$ onto~a~group from~$\mathcal{C}$ such that the~elements~$x\sigma$ and~$y\sigma$ are~still not~equal (respectively, not~conjugate in~$X\sigma$). A~generalized Baumslag--Solitar group or~GBS-group is~the~fundamental group of~a~finite connected graph of~groups whose all vertex and~edge groups are~infinite cyclic. An~ordinary Baumslag--Solitar group is~the~GBS-group that corresponds to~a~graph containing only one vertex and~one loop. Suppose that the~class~$\mathcal{C}$ consists of~periodic groups and~is~closed under~taking subgroups and~unrestricted wreath products. We~prove that a~non-solv\-able GBS-group is~conjugacy $\mathcal{C}$\nobreakdash-sepa\-ra\-ble if and~only~if it~is~residually a~$\mathcal{C}$\nobreakdash-group. We~also find a~criterion for~a~solvable GBS-group to~be~conjugacy $\mathcal{C}$\nobreakdash-sepa\-ra\-ble. As~a~corollary, we~prove that an~arbitrary GBS-group is~conjugacy (finite) separable if and~only~if it~is~residually finite.
\end{abstract}


\thanks{The~study was~supported by~the~Russian Science Foundation grant No.~24-21-00307,\\ \url{http://rscf.ru/en/project/24-21-00307/}}

\vspace*{7pt}

\maketitle

\vspace*{-12pt}

\section{Introduction. Statement of~results}\label{es01}

Recall that a~\emph{generalized Baumslag--Solitar group} (or,~more briefly, a~\emph{GBS-group}) is~the fundamental group of~a~finite connected graph of~groups whose all vertex and~edge groups are~infinite cyclic (the~definition of~the~fundamental group of~a~graph of~groups can be~found in~Section~\ref{es03}). Although fundamental groups of~this type have been studied since the~early~1990s, the~term ``generalized Baumslag--Solitar group'' came into~use only in~the~2000s (see~\cite{Forester2002GT, Forester2003CMH, Forester2006GD, Levitt2007GT}).

If~the~graph of~groups contains only one vertex and~one loop, then the~corresponding GBS-group turns~out to~be~an~ordinary \emph{Baumslag--Solitar group}~\cite{BaumslagSolitar1962BAMS}, i.e.,~the~group defined by~the~presentation of~the~form
\[
\mathrm{BS}(m,n) = \langle a, t;\ t^{-1}a^{m}t = a^{n} \rangle,
\]
where $m \ne 0 \ne n$. Since the~groups $\mathrm{BS}(m,n)$, $\mathrm{BS}(n,m)$, $\mathrm{BS}(-m,-n)$, and~$\mathrm{BS}(-n,-m)$ are~pairwise isomorphic, we~can assume without loss of~generality that $0 < m \leqslant |n|$ in~the~above presentation.

A~GBS-group is~said to~be~\emph{elementary} if it~is~isomorphic to~an~infinite cyclic group or~the~group $\mathrm{BS}(1,n)$, where $|n| = 1$~\cite[p.~6]{Levitt2015JGT}. It~is~known~\cite{DelgadoRobinsonTimm2014AC} that a~GBS-group is~solvable if and~only~if it~is~elementary or~isomorphic to~the~group $\mathrm{BS}(1,n)$, where $|n| > 1$.

Generalized Baumslag--Solitar groups have been the~subject of~numerous investigations over~the~last thirty years (see, for~example,~\cite{Beeker2015GGD, ClayForester2008AGT, CornulierValette2015GD, DelgadoRobinsonTimm2011JPAA, DelgadoRobinsonTimm2018CAG, Dudkin2015JGT, Dudkin2016AL, Dudkin2017AL, Dudkin2018SMJ, Dudkin2020AL, Dudkin2021CA, Dudkin2022AL, DudkinMamontov2020JKT, DudkinYan2023SMJ, GandiniMeinertRuping2015GGD, Meinert2015AGT, Robinson2011RSMUP}). In~particular, residual properties of~such groups are~considered in~\cite{Robinson2010NM, Levitt2015JGT, Dudkin2020AM, Sokolov2021JA}. In~this article, we~study the~conjugacy separability of~GBS-groups.

Suppose that $\mathcal{C}$ is~a~class of~groups, $X$~is~a~group, and~$\theta$ is~a~relation between elements and~(or) subsets of~the~group~$X$, which is~defined in~each homomorphic image of~$X$. The~group~$X$ is~said to~be~\emph{residually a~$\mathcal{C}$\nobreakdash-group with~respect to~$\theta$} if, for~any elements and~subsets of~$X$ that are~not~in~$\theta$, there exists a~homomorphism~$\sigma$ of~$X$ onto~a~group from~$\mathcal{C}$ (a~$\mathcal{C}$\nobreakdash-group) such that the~images of~these elements and~subsets under~$\sigma$ are~still not~in~$\theta$. If~$\theta$ is~the~relation of~equality of~two elements, then it~is~usually not~mentioned and~$X$ is~simply called \emph{residually a~$\mathcal{C}$\nobreakdash-group}~\cite{Hall1954PLMS}. In~particular, if~$\mathcal{C}$ coincides with~the~class~$\mathcal{F}$ of~all finite groups, then $X$ is~said to~be~\emph{residually finite}. When $\theta$ is~the~relation of~conjugacy of~two elements, $X$~is~called a~\emph{conjugacy $\mathcal{C}$\nobreakdash-sepa\-ra\-ble group}~\cite{RibesZalesskii2016JGT} or~simply a~\emph{conjugacy separable} one if $\mathcal{C} = \mathcal{F}$. Since equal elements are~obviously conjugate, it~follows from~the~conjugacy $\mathcal{C}$\nobreakdash-sepa\-ra\-bil\-ity of~a~group that the~latter is~residually a~$\mathcal{C}$\nobreakdash-group. The~converse is~not~always true, and~this fact is~confirmed, in~particular, by~Theorem~\ref{et02} given below (see the~last paragraph of~this section for~details).

Let~us say that a~class of~groups is~\emph{non-trivial} if~it~contains at~least one non-trivial group. A~non-trivial class of~groups~$\mathcal{C}$ is~called a~\emph{root class} if~it~is~closed under~taking subgroups and~satisfies any of~the~following three equivalent conditions~(see~\cite{Sokolov2015CA}):

1)\hspace{1ex}given a~group~$X$ and~a~subnormal series $1 \leqslant Z \leqslant Y \leqslant X$, it~follows from~the~inclusions $X/Y \in \mathcal{C}$ and~$Y/Z \in \mathcal{C}$ that there exists a~normal subgroup~$T$ of~$X$ such that $T \leqslant Z$ and~$X/T \in \mathcal{C}$ (the~Gruenberg condition~\cite{Gruenberg1957PLMS});

2)\hspace{1ex}the~class~$\mathcal{C}$ is~closed under~taking unrestricted wreath products;

3)\hspace{1ex}the~class~$\mathcal{C}$ is~closed under~taking extensions and~unrestricted direct products of~the form $\prod_{y \in Y} X_{y}$, where $X,Y \in \mathcal{C}$ and~$X_{y}$ is~an~isomorphic copy of~$X$ for~each $y \in Y$.

The~examples of~root classes are~the~class of~all finite groups; the~class of~finite $p$\nobreakdash-groups, where $p$ is~a~prime; the~class of~periodic $\mathfrak{P}$\nobreakdash-groups of~finite exponent, where $\mathfrak{P}$ is~a~non-empty set of~primes; the~classes of~all solvable groups and~all torsion-free groups. It~is~easy to~see that a~non-trivial intersection of~any number of~root classes is~again a~root class. It~is~also clear that a~non-trivial class of~groups consisting only of~finite groups is~a~root class if and~only~if it~is~closed under~taking subgroups and~extensions.

The~use of~the~concept of~a~root class turned~out to~be~very productive in~the~study of~some residual properties of~free constructions of~groups: generalized free and~tree products, HNN-extensions, fundamental groups of~graphs of~groups,~etc. It~allows one to~prove several statements at~once and~quickly complicate the~constructions under~consideration; see, for~example,~\cite{Sokolov2015CA, Tumanova2017SMJ, Tumanova2019SMJ, SokolovTumanova2020LJM, Sokolov2021JA, Sokolov2022CA, Sokolov2023JGT, Sokolov2024SMJ}. At~the~same time, if~$\mathcal{C}$ is~a~root class different from~the~class of~all finite groups, very few facts are~known about~the~conjugacy $\mathcal{C}$\nobreakdash-sepa\-ra\-bil\-ity of~free constructions of~groups. Most of~the~results of~such type are~contained in~\cite{Ivanova2004MN, Ivanova2005BIvSU, Sokolov2015MN, ChagasZalesskii2015AM, Ferov2016JA, Moldavanskii2018CA}, and~this article complements this list.

Let $\mathfrak{P}$ be~a~set of~primes. Recall that an~integer is~said to~be~a~\emph{$\mathfrak{P}$\nobreakdash-num\-ber} if~all its prime divisors belong to~$\mathfrak{P}$; a~periodic group is~said to~be~a~\emph{$\mathfrak{P}$\nobreakdash-group} if~the~orders of~all its elements are~$\mathfrak{P}$\nobreakdash-num\-bers. Given a~class of~groups~$\mathcal{C}$ consisting only of~periodic groups, let~us denote by~$\mathfrak{P}(\mathcal{C})$ the~set of~primes defined as~follows: $p \in \mathfrak{P}(\mathcal{C})$ if and~only~if $p$~divides the~order of~an~element of~some $\mathcal{C}$\nobreakdash-group.

In~\cite{Tumanova2017SMJ, Sokolov2021JA}, a~criterion was found for~a~GBS-group to~be~residually a~$\mathcal{C}$\nobreakdash-group, where $\mathcal{C}$ was a~root class of~groups consisting only of~periodic groups. The~main result of~this article~is

\begin{etheorem}\label{et01}
Suppose that $\mathfrak{G}$ is~a~GBS-group\textup{,} $\mathcal{C}$~is~a~root class of~groups consisting only of~periodic groups\textup{,} and~$\mathcal{FS}_{\mathfrak{P}(\mathcal{C})}$ is~the~class of~all finite solvable $\mathfrak{P}(\mathcal{C})$\nobreakdash-groups. Suppose also that the~group~$\mathfrak{G}$ is~either elementary or~non-solv\-able. Then the~following statements are~equivalent.

\textup{1.\hspace{1ex}}The~group~$\mathfrak{G}$ is~residually a~$\mathcal{C}$\nobreakdash-group.

\textup{2.\hspace{1ex}}The~group~$\mathfrak{G}$ is~conjugacy $\mathcal{C}$\nobreakdash-sepa\-ra\-ble.

\textup{3.\hspace{1ex}}The~group~$\mathfrak{G}$ is~conjugacy $\mathcal{FS}_{\mathfrak{P}(\mathcal{C})}$\nobreakdash-sepa\-ra\-ble.
\end{etheorem}

Theorem~\ref{et01} does~not cover only the~case when $\mathfrak{G}$ is~isomorphic to~the~group $\mathrm{BS}(1,n)$ for~some $n \in \mathbb{Z} \setminus \{0, \pm 1\}$. In~this case, a~criterion for~$\mathfrak{G}$ to~be~conjugacy $\mathcal{C}$\nobreakdash-sepa\-ra\-ble is~given~by

\begin{etheorem}\label{et02}
Suppose that $\mathcal{C}$ is~a~root class of~groups consisting only of~periodic groups\textup{,} $\mathcal{FS}_{\mathfrak{P}(\mathcal{C})}$~is~the~class of~all finite solvable $\mathfrak{P}(\mathcal{C})$\nobreakdash-groups\textup{,} and~$n \in \mathbb{Z} \setminus \{0, \pm 1\}$. Then the~following statements are~equivalent.

\textup{1.\hspace{1ex}}The~group $\mathrm{BS}(1,n)$ is~conjugacy $\mathcal{C}$\nobreakdash-sepa\-ra\-ble.

\textup{2.\hspace{1ex}}The~group $\mathrm{BS}(1,n)$ is~conjugacy $\mathcal{FS}_{\mathfrak{P}(\mathcal{C})}$\nobreakdash-sepa\-ra\-ble.

\textup{3.\hspace{1ex}}The~set~$\mathfrak{P}(\mathcal{C})$ contains all prime numbers.
\end{etheorem}

The~formulated theorems imply

\begin{ecorollary}
An~arbitrary GBS-group is~conjugacy separable if and~only~if it~is~residually finite.
\end{ecorollary}

Theorems~\ref{et01} and~\ref{et02} generalize earlier results on~the~conjugacy separability of~ordinary Baumslag--Solitar groups: Theorems~10\nobreakdash--12 from~\cite{Moldavanskii2018CA} and~the~main result of~\cite{Moldavanskii2007BIvSU}. At~the~same time, the~proof of~Theorem~\ref{et02} essentially uses Theorem~10~from~\cite{Moldavanskii2018CA}.

Let~us make a~remark regarding the~criterion of~conjugacy separability from~Theorem~\ref{et02}. If~$\mathcal{C}$ is~a~root class of~groups consisting only of~periodic groups and~$n \in \mathbb{Z} \setminus \{0, \pm 1\}$, then the~group $\mathrm{BS}(1,n)$ is~residually a~$\mathcal{C}$\nobreakdash-group if and~only~if there exists a~number $p \in \mathfrak{P}(\mathcal{C})$ such that $p$ does~not~divide~$n$ and~the~order of~$n$ in~the~multiplicative group of~the~field~$\mathbb{Z}_{p}$ is~a~$\mathfrak{P}(\mathcal{C})$\nobreakdash-num\-ber~\cite{Tumanova2017SMJ}. In~particular, $\mathrm{BS}(1,n)$~is~residually a~finite $p$\nobreakdash-group for~some prime~$p$ if and~only~if $p$~divides $n-1$~\cite[Theorem~2]{Moldavanskii2018CA}. It~is~also known that, for~any prime~$p$ which does~not divide $n-1$, there exists a~prime $q > p$ such that $q$ does~not divide $n-1$ and~$\mathrm{BS}(1,n)$ is~residually an~$\mathcal{F}_{\{p,q\}}$\nobreakdash-group, where $\mathcal{F}_{\{p,q\}}$ is~the~class of~finite $\{p,q\}$\nobreakdash-groups~\cite[Corollary~1]{Moldavanskii2018CA}. Moreover, the~probability is~quite high that $\mathrm{BS}(1,n)$ is~residually an~$\mathcal{F}_{\{p,q\}}$\nobreakdash-group for~randomly chosen primes $p$ and~$q$ that do~not divide $n-1$ and~do~not exceed some given number~\cite{Tumanova2022ACCS}. Thus, there~are many examples where $\mathcal{C}$ is~a~root class of~groups consisting only of~periodic groups, $n \in \mathbb{Z} \setminus \{0, \pm 1\}$, and~$\mathrm{BS}(1,n)$ is~residually a~$\mathcal{C}$\nobreakdash-group, but~is~not~conjugacy $\mathcal{C}$\nobreakdash-sepa\-ra\-ble.

\vspace*{3pt}

\section{Solvable GBS-groups. Proof of~Theorem~\ref{et02}}\label{es02}

In~this article, we~use the~following notation:

\smallskip

\makebox[8.5ex][l]{$\langle x \rangle$}the~cyclic group generated by~an~element~$x$;

\makebox[8.5ex][l]{$(x,y)$}the~greatest common divisor of~numbers $x$ and~$y$;

\makebox[8.5ex][l]{$[x,y]$}the~commutator of~elements $x$ and~$y$, which is~equal to~$x^{-1}y^{-1}xy$;

\makebox[8.5ex][l]{$x \mid y$}a~number~$x$ divides a~number~$y$ (in~the~sense defined below);

\makebox[8.5ex][l]{$x \sim_{Z} y$}elements $x$ and~$y$ are~conjugate in~a~group~$Z$;

\makebox[8.5ex][l]{$\mathbb{Z}_{\geqslant 0}$}the~set of~non-nega\-tive integers;

\makebox[8.5ex][l]{$\mathcal{F}_{\mathfrak{P}}$}the~class of~all finite $\mathfrak{P}$\nobreakdash-groups;

\makebox[8.5ex][l]{$\mathcal{FS}_{\mathfrak{P}}$}the~class of~all finite solvable $\mathfrak{P}$\nobreakdash-groups.

\smallskip

Throughout the~paper, we~assume that, if~$x,y \in \mathbb{Z}$, then $y \mid x$ if and~only~if there exists a~number $z \in \mathbb{Z}$ satisfying the~equality $x = yz$. In~particular, $0 \mid x$ if and~only~if $x = 0$.

For~any $n \in \mathbb{Z} \setminus \{0\}$, let
\[
\Omega(n) = \big\{(r,s) \in \mathbb{Z}^{2} \mid r > 0,\ s > 0,\ n^{r} \equiv 1\kern-.5em \pmod s\big\},
\]
and~if~$(r,s) \in \Omega(n)$, let
\[
H(n,r,s) = \langle t, a;\ t^{-1}at = a^{n},\ t^{r} = 1,\ a^{s} = 1 \rangle.
\]

\begin{eproposition}\label{ep01}
If~$n \in \mathbb{Z} \setminus \{0\}$ and~$(r,s) \in \Omega(n)$\textup{,} then the~following statements hold.

\textup{1.\hspace{1ex}}Any element of~the~group $\mathrm{BS}(1,n)$ can be~written in~the~form $t^{u}a^{w}t^{-v}$ for~suitable numbers $u,v \in \mathbb{Z}_{\geqslant 0}$ and~$w \in \mathbb{Z}$. In~particular\textup{,} this element is~conjugate to~an~element of~the~form~$t^{u}a^{v}$\textup{,} where $u,v \in \mathbb{Z}$. Since the~presentation of~$\mathrm{BS}(1,n)$ is~a~part of~the~presentation of~$H(n,r,s)$\textup{,} the~same statements also hold for~the~group~$H(n,r,s)$.

\pagebreak

\textup{2.\hspace{1ex}}Given $u \in \mathbb{Z}_{\geqslant 0}$ and~$v,w \in \mathbb{Z}$\textup{,} the~elements $t^{u}a^{v}$ and~$t^{u}a^{w}$ are~conjugate in~$\mathrm{BS}(1,n)$ \textup{(}in~$H(n,r,s)$\textup{)} if and~only~if there exist $x,y \in \mathbb{Z}_{\geqslant 0}$ such that $n^{u}-1 \mid vn^{x} - wn^{y}$ \textup{(}respectively\textup{,} $(n^{u}-1,\;s) \mid vn^{x} - wn^{y}$\textup{)}.
\end{eproposition}

\begin{proof}
1.\hspace{1ex}It~follows from~the~relation $t^{-1}at = a^{n}$ that $a^{k}t = ta^{nk}$ and~$t^{-1}a^{k} = a^{nk}t^{-1}$ for~any $k \in \mathbb{Z}$. Clearly, these equalities allow~us to~transform any element of~the~group $\mathrm{BS}(1,n)$ to~the~required form.

\smallskip

2.\hspace{1ex}By~Statement~1, if~$X = \mathrm{BS}(1,n)$ or~$X = H(n,r,s)$, then $t^{u}a^{v} \sim_{X} t^{u}a^{w}$ if and~only~if there exist $x,y \in \mathbb{Z}_{\geqslant 0}$ and~$z \in \mathbb{Z}$ such that $(t^{x}a^{z}t^{-y})^{-1}t^{u}a^{v}(t^{x}a^{z}t^{-y}) = t^{u}a^{w}$. We~have
\begin{align*}
t^{y}a^{-z}t^{-x}t^{u}a^{v}t^{x}a^{z}t^{-y} = t^{u}a^{w} 
&\Leftrightarrow t^{u}(t^{-u}a^{-z}t^{u})(t^{-x}a^{v}t^{x})a^{z} = t^{u}(t^{-y}a^{w}t^{y}) \\
&\Leftrightarrow a^{vn^{x} - z(n^{u}-1)} = a^{wn^{y}} \\
&\Leftrightarrow \begin{cases}
vn^{x} - (n^{u}-1)z = wn^{y} &\text{if}\ X = \mathrm{BS}(1,n), \\
vn^{x} - (n^{u}-1)z \equiv wn^{y}\kern-.5em \pmod s &\text{if}\ X = H(n,r,s).
\end{cases}
\end{align*}
It~remains to~note that
\[
\big(\exists z \in \mathbb{Z}\ [vn^{x} - (n^{u}-1)z = wn^{y}]\big) 
\Leftrightarrow \big(n^{u}-1 \mid vn^{x} - wn^{y}\big)
\]
and, since $(n^{u}-1,\;s) = (n^{u}-1)p + sq$ for~some $p,q \in \mathbb{Z}$,
\begin{align*}
\exists z \in \mathbb{Z}\ [vn^{x} - (n^{u}-1)z \equiv wn^{y}\kern-.5em \pmod s] 
&\Leftrightarrow \exists z,\ell \in \mathbb{Z}\ [vn^{x} - (n^{u}-1)z = wn^{y} + s\ell] \\
&\Leftrightarrow \exists m \in \mathbb{Z}\ [vn^{x} - (n^{u}-1,\;s)m = wn^{y}] \\
&\Leftrightarrow (n^{u}-1,\;s) \mid vn^{x} - wn^{y}.
\end{align*}

\vspace*{-18pt}

\mbox{}
\end{proof}

Given a~number $n \in \mathbb{Z} \setminus \{0\}$ and~a~non-empty set of~primes~$\mathfrak{P}$, let $\Omega(n,\mathfrak{P})$ denote the~subset of~$\Omega(n)$ defined as~follows: $(r,s) \in \Omega(n,\mathfrak{P})$ if and~only~if $r$ and~$s$ are~$\mathfrak{P}$\nobreakdash-num\-bers and~$(r,s) \in \Omega(n)$. Let also
\[
\Xi(n,\mathfrak{P}) = \big\{s \in \mathbb{Z} \mid 
\exists r \in \mathbb{Z}\ [(r,s) \in \Omega(n,\mathfrak{P})]\big\}.
\]

\begin{eproposition}\label{ep02}
If $n \in \mathbb{Z} \setminus \{0\}$ and~$\mathfrak{P}$ is~a~non-empty set of~primes\textup{,} then the~set~$\Xi(n,\mathfrak{P})$ is~closed under~the~operations of~multiplication and~taking a~positive divisor.
\end{eproposition}

\begin{proof}
Let $s \in \Xi(n,\mathfrak{P})$. By~the~definition of~$\Xi(n,\mathfrak{P})$, $s$~is~a~positive $\mathfrak{P}$\nobreakdash-num\-ber and~there exists a~positive $\mathfrak{P}$\nobreakdash-num\-ber~$r$ such that $n^{r} \equiv 1 \pmod s$. Obviously, if~$s^{\prime} > 0$ and~$s^{\prime} \mid\nolinebreak s$, then $s^{\prime}$ is~a~$\mathfrak{P}$\nobreakdash-num\-ber and~$n^{r} \equiv 1 \pmod {s^{\prime}}$. Therefore, the~set~$\Xi(n,\mathfrak{P})$ is~closed under~taking positive divisors. Let~us show by~induction on~$k$ that $n^{rs^{k-1}} \equiv 1 \pmod {s^{k}}$ for~any $k \geqslant 1$. Indeed, the~basis of~induction holds. If~$\alpha_{k}$ denotes the~number~$n^{rs^{k-1}}$ and~$\alpha_{k} \equiv 1 \pmod {s^{k}}$ for~some $k \geqslant 1$, then
\[
1 + \alpha_{k}^{\vphantom{1}} + \alpha_{k}^{2} + \ldots + \alpha_{k}^{s-1} \equiv s \pmod {s^{k}}.
\]
It~follows that $s \mid 1 + \alpha_{k}^{\vphantom{1}} + \alpha_{k}^{2} + \ldots + \alpha_{k}^{s-1}$ and,~therefore,
\[
s^{k+1} \mid (\alpha_{k}^{\vphantom{1}} - 1)(1 + \alpha_{k}^{\vphantom{1}} + \alpha_{k}^{2} + \ldots + \alpha_{k}^{s-1}) = n^{rs^{k}} - 1.
\]

Suppose now that $r_{1}$, $r_{2}$, $s_{1}$, and~$s_{2}$ are~positive $\mathfrak{P}$\nobreakdash-num\-bers such that $n^{r_{1}} \equiv 1 \pmod {s_{1}}$ and~$n^{r_{2}} \equiv 1 \pmod {s_{2}}$. Let $d$ denote the~greatest divisor of~$s_{1}$ that is~coprime with~$s_{2}$, and~let $\ell = s_{1} / d$. Then all prime divisors of~$\ell$ divide~$s_{2}^{\vphantom{k}}$ and,~therefore, $\ell \mid s_{2}^{k}$ for~some $k \geqslant 1$. As~proven above, $n^{r_{2}^{\vphantom{k}}s_{2}^{k}} \equiv 1 \pmod {s_{2}^{k+1}}$. The~relations $n^{r_{1}} \equiv 1 \pmod {s_{1}}$ and~$d \mid\nolinebreak s_{1}$ mean that $n^{r_{1}} \equiv 1 \pmod d$. Hence, $n^{r_{1}^{\vphantom{k}}r_{2}^{\vphantom{k}}s_{2}^{k}} \equiv 1 \pmod {s_{2}^{k+1}}$, $n^{r_{1}^{\vphantom{k}}r_{2}^{\vphantom{k}}s_{2}^{k}} \equiv 1 \pmod d$,~and $n^{r_{1}^{\vphantom{k}}r_{2}^{\vphantom{k}}s_{2}^{k}} \equiv 1 \pmod {ds_{2}^{k+1}}$ because $(d, s_{2}^{\vphantom{k}}) = 1$. Since $s_{1}^{\vphantom{k}}s_{2}^{\vphantom{k}} = d\ell s_{2}^{\vphantom{k}} \mid ds_{2}^{k+1}$, it~follows that $n^{r_{1}^{\vphantom{k}}r_{2}^{\vphantom{k}}s_{2}^{k}} \equiv 1 \pmod {s_{1}^{\vphantom{k}}s_{2}^{\vphantom{k}}}$. It~remains to~note that $s_{1}^{\vphantom{k}}s_{2}^{\vphantom{k}}$ and~$r_{1}^{\vphantom{k}}r_{2}^{\vphantom{k}}s_{2}^{k}$ are, obviously, positive $\mathfrak{P}$\nobreakdash-num\-bers and,~therefore, $s_{1}^{\vphantom{k}}s_{2}^{\vphantom{k}} \in \Xi(n,\mathfrak{P})$.
\end{proof}

\begin{eproposition}\label{ep03}
The~following statements hold.

\textup{1.\hspace{1ex}}If~$n \in \mathbb{Z} \setminus \{0\}$ and~$(r,s) \in \Omega(n)$\textup{,} then $H(n,r,s)$ is~a~finite solvable group of~order~$rs$.

\textup{2.\hspace{1ex}}Suppose that $\mathcal{C}$ is~a~root class of~groups consisting only of~periodic groups and~$n \in \mathbb{Z} \setminus \{0\}$. Then every homomorphism of~$\mathrm{BS}(1,n)$ onto~a~group from~$\mathcal{C}$ continues the~natural homomorphism $\eta\colon \mathrm{BS}(1,n) \to H(n,r,s)$ for~some $(r,s) \in \Omega(n,\mathfrak{P}(\mathcal{C}))$. In~particular\textup{,} if~the~group $\mathrm{BS}(1,n)$ is~conjugacy $\mathcal{C}$\nobreakdash-sepa\-ra\-ble\textup{,} then it~is~conjugacy $\mathcal{FS}_{\mathfrak{P}(\mathcal{C})}$\nobreakdash-sepa\-ra\-ble.
\end{eproposition}

\begin{proof}
1.\hspace{1ex}It~follows from~the~equalities $t^{-1}at = a^{n}$ and~$tat^{-1} = t^{-(r-1)}at^{r-1} = a^{n^{r-1}}$ that the~subgroup~$\langle a \rangle$ is~normal in~the~group $H(n,r,s)$ and,~therefore, the~latter is~a~split extension of~this subgroup by~the~group $\langle t;\ t^{r} = 1 \rangle$. Let
\[
\mathfrak{A} = \langle \mathfrak{a};\ \mathfrak{a}^{s} = 1 \rangle
\quad\text{and}\quad 
\mathfrak{T} = \langle \mathfrak{t};\ \mathfrak{t}^{r} = 1 \rangle.
\]
The~inclusion $(r,s) \in \Omega(n)$ means that $n^{r} \equiv 1 \pmod s$ and,~in~particular, $(n,s) = 1$. Therefore, the~endomorphism of~$\mathfrak{A}$ taking~$\mathfrak{a}$ to~$\mathfrak{a}^{n}$ is~an~automorphism and~its order divides~$r$. This allows~us to~consider the~split extension~$\mathfrak{S}$ of~$\mathfrak{A}$ by~$\mathfrak{T}$ such that $\mathfrak{t}^{-1}\mathfrak{a}\mathfrak{t} = \mathfrak{a}^{n}$. Let $\xi$ be~the~mapping of~words in~the~alphabet $\{t, t^{-1}, a, a^{-1}\}$ that acts on~the~symbols of~this alphabet in~accordance with~the~rule
\[
t \mapsto \mathfrak{t},\quad 
t^{-1} \mapsto \mathfrak{t}^{-1},\quad 
a \mapsto \mathfrak{a},\quad 
a^{-1} \mapsto \mathfrak{a}^{-1}.
\]
It~is~obvious that $\xi$ takes all defining relations of~$H(n,r,s)$ into~the~equalities valid in~$\mathfrak{S}$ and,~therefore, defines a~homomorphism~$\varphi$ of~the~first group into~the~second one. Since $\langle a \rangle\varphi = \mathfrak{A}$, the~order of~the~subgroup $\langle a \rangle$ equals~$s$. Hence, the~order of~the~group $H(n,r,s)$ is~equal~to~$rs$.

\smallskip

2.\hspace{1ex}Let $\sigma$ be~a~homomorphism of~$\mathrm{BS}(1,n)$ onto~a~$\mathcal{C}$\nobreakdash-group~$X$. Since $\mathcal{C}$ consists of~periodic groups, the~orders $r$ and~$s$ of~the~elements $t\sigma$ and~$a\sigma$ are~finite and~are~$\mathfrak{P}(\mathcal{C})$\nobreakdash-num\-bers. The~equalities $(t\sigma)^{-1}a\sigma t\sigma = (a\sigma)^{n}$ and~$(t\sigma)^{r} = 1$ imply that $a\sigma =\nolinebreak (t\sigma)^{-r}a\sigma (t\sigma)^{r} =\nolinebreak (a\sigma)^{n^{r}}$. Since the~order of~$a\sigma$ is~equal to~$s$, it~follows that $n^{r}\kern-1.5pt{} \equiv\kern-1pt{} 1\kern-2.5pt{} \pmod {\mbox{}\kern-1pt{}s}$. Thus, $(r,s) \in\nolinebreak \Omega(n,\mathfrak{P}(\mathcal{C}))$ and~we~can consider the~group $H(n,r,s)$. As~above, if~$\xi$ is~the~mapping of~words in~the~alphabet $\{t^{\pm 1}, a^{\pm 1}\}$ that acts on~the~symbols of~this alphabet by~the~rule $t^{\pm 1} \mapsto (t\sigma)^{\pm 1}$, $a^{\pm 1} \mapsto (a\sigma)^{\pm 1}$, then $\xi$ takes all defining relations of~$H(n,r,s)$ into~the~equalities valid in~$X$ and,~therefore, defines a~homomorphism $\rho\colon H(n,r,s) \to X$. By~the~construction, the~homomorphisms $\sigma$ and~$\eta\rho$ act identically on~the~generators of~$\mathrm{BS}(1,n)$. Hence, $\sigma = \eta\rho$ and~$H(n,r,s)$ is~the~desired group.
\end{proof}

\begin{eproposition}\label{ep04}
\textup{\cite[Proposition~8]{Tumanova2019SMJ}}
Let $\mathcal{C}$ be~a~root class of~groups consisting only of~periodic groups. A~finite solvable group belongs to~$\mathcal{C}$ if and~only~if it~is~a~$\mathfrak{P}(\mathcal{C})$\nobreakdash-group.
\end{eproposition}

\begin{eproposition}\label{ep05}
If $m,n \in \mathbb{Z}$ and~$0 < m \leqslant |n|$\textup{,} then the~following statements hold.

\textup{1.\hspace{1ex}}The~group $\mathrm{BS}(m,n)$ is~residually finite if and~only~if $m = 1$ or~$m = |n|$ \textup{\cite[Theorem~C]{Meskin1972TAMS}}.

\textup{2.\hspace{1ex}}If~the~group $\mathrm{BS}(m,n)$ is~residually finite\textup{,} then it~is~conjugacy separable \textup{\cite[Theorem~10]{Moldavanskii2018CA}}.

\textup{3.\hspace{1ex}}Given a~root class of~groups~$\mathcal{C}$\kern-2pt{} consisting only of~periodic groups\textup{,}\kern-1pt{} the~group $\mathrm{BS}(m,-m)$ is~residually a~$\mathcal{C}$\nobreakdash-group if and~only~if $m$~is~a~$\mathfrak{P}(\mathcal{C})$\nobreakdash-num\-ber and~$2 \in \mathfrak{P}(\mathcal{C})$~\textup{\cite{Tumanova2017SMJ}}.
\end{eproposition}

\begin{eproposition}\label{ep06}
Suppose that $\mathcal{C}$ is~a~root class of~groups consisting only of~periodic groups and~$n \in \mathbb{Z} \setminus \{0\}$. Then the~following statements hold.

\textup{1.\hspace{1ex}}The~group $\mathrm{BS}(1,n)$ is~conjugacy $\mathcal{C}$\nobreakdash-sepa\-ra\-ble if and~only~if
\begin{equation}\label{ef01}
\forall u \in \mathbb{Z}_{\geqslant 0}\ 
\forall v,w \in \mathbb{Z}\ \Bigg[
\begin{aligned}
&\big(\forall x,y \in \mathbb{Z}_{\geqslant 0}\ 
[n^{u}-1 \nmid vn^{x} - wn^{y}]\big) \Rightarrow \\
&\big(\exists s \in \Xi(n,\mathfrak{P}(\mathcal{C}))\ 
\forall x,y \in \mathbb{Z}_{\geqslant 0}\ 
[(n^{u}-1,\;s) \nmid vn^{x} - wn^{y}]\big)
\end{aligned}
\Bigg].
\end{equation}

\textup{2.\hspace{1ex}}If~the~group $\mathrm{BS}(1,-1)$ is~residually a~$\mathcal{C}$\nobreakdash-group\textup{,} then it~is~conjugacy $\mathcal{FS}_{\mathfrak{P}(\mathcal{C})}$\nobreakdash-sepa\-rable.
\end{eproposition}

\begin{proof}
1.\hspace{1ex}Let
\[
\mathcal{H} = \big\{H(n,r,s) \mid 
(r,s) \in \Omega(n,\mathfrak{P}(\mathcal{C}))\big\}.
\]
It~follows from~Propositions~\ref{ep03} and~\ref{ep04} that $\mathcal{H} \subseteq \mathcal{FS}_{\mathfrak{P}(\mathcal{C})} \subseteq \mathcal{C}$ and~the~group $\mathrm{BS}(1,n)$ is~conjugacy $\mathcal{C}$\nobreakdash-sepa\-ra\-ble if and~only~if it~is~conjugacy $\mathcal{H}$\nobreakdash-sepa\-ra\-ble. By~Proposition~\ref{ep01}, any element of~$\mathrm{BS}(1,n)$ is~conjugate to~an~element of~the~form~$t^{u}a^{v}$, where $u,v \in \mathbb{Z}$. Suppose that $g = t^{u}a^{v}$, $g^{\prime} = t^{u^{\prime}}a^{w}$, $u \ne u^{\prime}$, and~$r$ is~a~$\mathfrak{P}(\mathcal{C})$\nobreakdash-num\-ber such that $r > |u| + |u^{\prime}|$. Then $(r,1) \in \Omega(n,\mathfrak{P}(\mathcal{C}))$ and~the~natural homomorphism of~$\mathrm{BS}(1,n)$ onto~the~finite cyclic group $H(n,r,1)$ takes $g$ and~$g^{\prime}$ to~non-equal and,~therefore, non-con\-ju\-gate elements. Thus, $\mathrm{BS}(1,n)$ is~conjugacy $\mathcal{H}$\nobreakdash-sepa\-ra\-ble if and~only~if
\begin{equation}\label{ef02}
\forall u,v,w \in \mathbb{Z}\ 
\Big[\big(t^{u}a^{v} \nsim_{BS(1,n)} t^{u}a^{w}\big) \Rightarrow 
\big(\exists (r,s) \in \Omega(n,\mathfrak{P}(\mathcal{C}))\ 
[t^{u}a^{v} \nsim_{H(n,r,s)} t^{u}a^{w}]\big)\Big].
\end{equation}

Let $X$ denote any of~the~groups $\mathrm{BS}(1,n)$ and~$H(n,r,s)$. Then, if $u < 0$, the~equalities
\[
(t^{u}a^{v})^{-1} = t^{-u}t^{u}a^{-v}t^{-u} = t^{-u}a^{-vn^{-u}}
\]
hold in~$X$, and~$t^{u}a^{v} \sim_{X} t^{u}a^{w}$ if and~only~if $(t^{u}a^{v})^{-1} \sim_{X} (t^{u}a^{w})^{-1}$. Therefore, the~number~$u$ in~\eqref{ef02} can be~considered non-nega\-tive. Since
\[
\big(\exists (r,s) \in \Omega(n,\mathfrak{P}(\mathcal{C}))\big) 
\Leftrightarrow 
\big(\exists s \in \Xi(n,\mathfrak{P}(\mathcal{C}))\big),
\]
it~follows that~\eqref{ef01} and~\eqref{ef02} are~equivalent due~to~Statement~2 of~Proposition~\ref{ep01}.

\smallskip

2.\hspace{1ex}Let~us show that if $n = -1$ and~$\mathrm{BS}(1,-1)$ is~residually a~$\mathcal{C}$\nobreakdash-group, then~\eqref{ef02} holds and,~therefore, the~group $\mathrm{BS}(1,-1)$ is~conjugacy $\mathcal{FS}_{\mathfrak{P}(\mathcal{C})}$\nobreakdash-sepa\-ra\-ble due~to~the~above arguments.

Suppose that $u,v,w \in \mathbb{Z}$, $t^{u}a^{v} \nsim_{BS(1,-1)} t^{u}a^{w}$, and~$s$ is~a~$\{2\}$\nobreakdash-num\-ber which is~greater than $|v|+|w|$. Since $\mathrm{BS}(1,-1)$ is~residually a~$\mathcal{C}$\nobreakdash-group, it~follows from~Statement~3 of~Proposition~\ref{ep05} that $2 \in \mathfrak{P}(\mathcal{C})$ and,~hence, $(2,s) \in \Omega(-1,\mathfrak{P}(\mathcal{C}))$. Since any element of~the~splitting extension $\mathrm{BS}(1,-1)$ can be~written in~the~form~$t^{x}a^{y}$ for~some $x,y \in \mathbb{Z}$~and
\[
(t^{x}a^{y})^{-1}t^{u}a^{v}(t^{x}a^{y}) = t^{u}a^{(-1)^{|x|}v + (1-(-1)^{|u|})y},
\]
the~conjugacy class~$Y$ of~$t^{u}a^{v}$ in~$\mathrm{BS}(1,-1)$ is~either the~set $\{t^{u}a^{v}, t^{u}a^{-v}\}$ (if~$u$ is~even) or~the~set $\{t^{u}a^{v+2k} \mid k \in \mathbb{Z}\}$ (if~$u$ is~odd). Clearly, if~$\eta\colon \mathrm{BS}(1,-1) \to H(-1,2,s)$ is~the~natural homomorphism, then the~conjugacy class of~$t^{u}a^{v} = (t^{u}a^{v})\eta$ in~$H(-1,2,s)$ coincides with~$Y\eta$. Since $|w \pm v| \leqslant |v|+|w| < s$ and~$s$ is~even, it~follows from~the~relation $t^{u}a^{w} \notin Y$ that $(t^{u}a^{w})\eta \notin Y\eta$. Hence, $t^{u}a^{v} \nsim_{H(-1,2,s)} t^{u}a^{w}$ and~\eqref{ef02} holds.
\end{proof}

\begin{eproposition}\label{ep07}
Suppose that $\mathcal{C}$ is~a~root class of~groups consisting only of~periodic groups and~$\mathfrak{G}$ is~a~solvable GBS-group. Then the~following statements are~equivalent.

\textup{1.\hspace{1ex}}The~group~$\mathfrak{G}$ is~conjugacy $\mathcal{C}$\nobreakdash-sepa\-ra\-ble.

\textup{2.\hspace{1ex}}The~group~$\mathfrak{G}$ is~conjugacy $\mathcal{F}_{\mathfrak{P}(\mathcal{C})}$\nobreakdash-sepa\-ra\-ble.

\textup{3.\hspace{1ex}}The~group~$\mathfrak{G}$ is~conjugacy $\mathcal{FS}_{\mathfrak{P}(\mathcal{C})}$\nobreakdash-sepa\-ra\-ble.
\end{eproposition}

\begin{proof}
The~equivalence of~Statements~2 and~3 follows from~the~fact that every homomorphic image of~$\mathfrak{G}$ is~a~solvable group. The~implication~$3 \Rightarrow 1$ holds due~to~Proposition~\ref{ep04}. The~implication~$1 \Rightarrow 3$ is~obvious if~$\mathfrak{G} \cong \mathbb{Z}$ and~follows from~Proposition~\ref{ep03} if~$\mathfrak{G} \cong \mathrm{BS}(1,n)$ for~some $n \in \mathbb{Z} \setminus \{0\}$.
\end{proof}

\begin{proof}[\textup{\textbf{Proof of~Theorem~\ref{et02}}}]
If~the~set $\mathfrak{P}(\mathcal{C})$ contains all primes, then the~group $\mathrm{BS}(1,n)$ is~conjugacy $\mathcal{F}_{\mathfrak{P}(\mathcal{C})}$\nobreakdash-sepa\-ra\-ble due~to~Statements~1 and~2 of~Proposition~\ref{ep05}. Therefore, the~implication~$3 \Rightarrow 2$, as~well as~the~implication~$2 \Rightarrow 1$, follows from~Proposition~\ref{ep07}. Let~us show that if $u \geqslant 2$, $v = n^{u}-1$, and~$|v| \notin \Xi(n,\mathfrak{P}(\mathcal{C}))$, then~\eqref{ef01} does~not hold and,~hence, the~group $\mathrm{BS}(1,n)$ is~not~conjugacy $\mathcal{C}$\nobreakdash-sepa\-ra\-ble due~to~Proposition~\ref{ep06}.

Indeed, since $|v| \notin \Xi(n,\mathfrak{P}(\mathcal{C}))$, $|n| \geqslant 2$, and~$|v| \geqslant |n|^{u}-1 \geqslant 3$, Proposition~\ref{ep02} implies the~existence of~a~prime divisor~$q$ of~$v$ such that $q \notin \Xi(n,\mathfrak{P}(\mathcal{C}))$. Let $w = v / q$. It~follows from~the~relations $v = n^{u}-1$, $u > 0$, and~$q \mid v$ that $q \nmid n$. Hence, $qw \nmid wn^{y}$ for~any $y \in \mathbb{Z}_{\geqslant 0}$ and,~therefore, $v = qw \nmid vn^{x} - wn^{y}$ for~all $x,y \in \mathbb{Z}_{\geqslant 0}$. At~the~same time, if~$s \in \Xi(n,\mathfrak{P}(\mathcal{C}))$, then $q \nmid s$ because the~set~$\Xi(n,\mathfrak{P}(\mathcal{C}))$ is~closed under~taking positive divisors. It~follows that $(v,s) \mid w \mid vn^{x} - wn^{y}$ for~any $x,y \in \mathbb{Z}_{\geqslant 0}$ and,~thus,~\eqref{ef01} does~not hold.

Let~us prove the~implication~$1 \Rightarrow 3$ by~contradiction and~assume that the~group $\mathrm{BS}(1,n)$ is~conjugacy $\mathcal{C}$\nobreakdash-sepa\-ra\-ble, but~the~set~$\mathfrak{P}(\mathcal{C})$ does~not~contain some prime number~$p$. If~$u = p^{2}$ and~$v = n^{p^{2}}-1 = n^{u}-1$, then $|v| \in \Xi(n,\mathfrak{P}(\mathcal{C}))$ as~proven above. It~follows from~the~definition of~the~set~$\Xi(n,\mathfrak{P}(\mathcal{C}))$ that there exists a~positive $\mathfrak{P}(\mathcal{C})$\nobreakdash-num\-ber~$r$ such that $n^{r} \equiv 1 \pmod {|v|}$. Let~us write~$r$ in~the~form $r = uk + \ell$, where $k \geqslant 0$ and~$0 \leqslant \ell < u$. Then $1 \equiv n^{r} = n^{uk+\ell} \equiv n^{\ell} \pmod {|v|}$ because $n^{u}-1 = v \equiv 0 \pmod {|v|}$. It~follows from~the~inequalities $|n| \geqslant 2$, $\ell < u$, and~$u \geqslant 4$ that
\[
|n^{\ell} - 1| \leqslant |n|^{\ell} + 1 < |n|^{u-1} + (|n|^{u-1} - 1) \leqslant |n|^{u} - 1 \leqslant |v|
\]
and,~therefore, $n^{\ell} = 1$. Hence, $\ell = 0$, $p^{2} = u \mid r$, and~since $r$ is~a~$\mathfrak{P}(\mathcal{C})$\nobreakdash-num\-ber, $p \in \mathfrak{P}(\mathcal{C})$, which contradicts the~assumption.
\end{proof}

\vspace*{0.5pt}

\section{Graphs of~groups and~their fundamental groups}\label{es03}

Suppose that $\Gamma$ is~a~non-empty connected undirected graph with~a~vertex set~$\mathcal{V}$ and an~edge set~$\mathcal{E}$ (loops and~multiple edges are~allowed). Let~us associate each vertex $v \in \mathcal{V}$ with~a~group~$G_{v}$, while each edge $e \in \mathcal{E}$ with~a~direction, a~group~$H_{e}$, and~injective homomorphisms $\varphi_{+e}\colon H_{e} \to G_{e(1)}$ and~$\varphi_{-e}\colon H_{e} \to G_{e(-1)}$, where $e(1)$ and~$e(-1)$ are~the~vertices that are~the~ends of~$e$. As~a~result, we~get a~directed \emph{graph of~groups}, which is~further denoted by~$\mathcal{G}(\Gamma)$. Let~us refer to~the~groups~$G_{v}$, the~groups~$H_{e}$, and~the~subgroups $H_{\varepsilon e} = H_{e}\varphi_{\varepsilon e}$, where $v \in \mathcal{V}$, $e \in \mathcal{E}$, and~$\varepsilon = \pm 1$, as~\emph{vertex groups}, \emph{edge groups}, and~\emph{edge subgroups} respectively.

The~\emph{fundamental group} of~the~graph~$\mathcal{G}(\Gamma)$ is~usually denoted by~the~symbol~$\pi_{1}(\mathcal{G}(\Gamma))$ and~can be~defined as~follows. Let $\mathcal{T}$ be~a~maximal subtree of~$\Gamma$, and~let $\mathcal{E}_{\mathcal{T}}$ be~the~edge set of~$\mathcal{T}$. Then the~generators of~$\pi_{1}(\mathcal{G}(\Gamma))$ are~the~generators of~the~groups~$G_{v}$, $v \in \mathcal{V}$, and~symbols~$t_{e}$, $e \in \mathcal{E} \setminus \mathcal{E}_{\mathcal{T}}$, while the~defining relations are~those of~$G_{v}$, $v \in \mathcal{V}$, and~also all possible relations of~the~forms
\begin{align*}
h_{e}^{\vphantom{1}}\varphi_{+e}^{\vphantom{1}} &= h_{e}^{\vphantom{1}}\varphi_{-e}^{\vphantom{1}} \quad 
(e \in \mathcal{E}_{\mathcal{T}}^{\vphantom{1}},\ h_{e}^{\vphantom{1}} \in H_{e}^{\vphantom{1}}), \\[-3pt]
t_{e}^{-1}(h_{e}^{\vphantom{1}}\varphi_{+e}^{\vphantom{1}})t_{e}^{\vphantom{1}} &= h_{e}^{\vphantom{1}}\varphi_{-e}^{\vphantom{1}} \quad 
(e \in \mathcal{E} \setminus \mathcal{E}_{\mathcal{T}}^{\vphantom{1}},\ 
h_{e}^{\vphantom{1}} \in H_{e}^{\vphantom{1}}), 
\end{align*}
where $h\varphi_{\varepsilon e}$, $\varepsilon = \pm 1$, is~a~word in~the~generators of~the~group~$G_{e(\varepsilon)}$ defining the~image of~$h_{e}$ under~$\varphi_{\varepsilon e}$~\cite[\S~5.1]{Serre1980}.

Obviously, the~presentation of~the~group~$\pi_{1}(\mathcal{G}(\Gamma))$ depends on~the~choice of~the~maximal subtree~$\mathcal{T}$. But~it~is~known that all such presentations define isomorphic groups~\cite[\S~5.1]{Serre1980}, and~this allows~us to~say about~the~fundamental group of~a~graph of~groups without specifying the~maximal subtree used to~construct its presentation. It~is~also known that, for~each vertex $v \in \mathcal{V}$, the~group~$G_{v}$ is~embedded into~$\pi_{1}(\mathcal{G}(\Gamma))$ via~the~identical mapping of~its generators and,~therefore, can be~considered a~subgroup of~$\pi_{1}(\mathcal{G}(\Gamma))$~\cite[\S~5.2]{Serre1980}. The~following statement is~a~special case of~Proposition~13 from~\cite{SokolovTumanova2020LJM}.

\begin{eproposition}\label{ep08}
Suppose that $\Gamma$ is~a~finite graph and~$N$ is~a~normal subgroup of~$\pi_{1}(\mathcal{G}(\Gamma))$. If~$N \cap G_{v} = 1$ for~every $v \in \mathcal{V}$\textup{,} then $N$ is~a~free group.
\end{eproposition}

If~all vertex and~edge groups of~$\mathcal{G}(\Gamma)$ are~infinite cyclic and~we~fix their generators $g_{v}$, $v \in \mathcal{V}$, and~$h_{e}$, $e \in \mathcal{E}$, then, for~any $e \in \mathcal{E}$, $\varepsilon = \pm 1$, the~homomorphism~$\varphi_{\varepsilon e}$ is~uniquely defined by~the~number $\lambda(\varepsilon e) \in \mathbb{Z} \setminus \{0\}$ such that $g_{e(\varepsilon)}^{\lambda(\varepsilon e)} = h_{\vphantom{(}e}^{\vphantom{(}}\varphi_{\vphantom{(}\varepsilon e}^{\vphantom{(}}$. Therefore, instead of~$\mathcal{G}(\Gamma)$, we~can consider the~directed \emph{graph with~labels~$\mathcal{L}(\Gamma)$} which is~obtained from~$\Gamma$ by~choosing a~direction for~each edge $e \in \mathcal{E}$ and~assigning non-zero integers $\lambda(+e)$ and~$\lambda(-e)$ to~the~ends $e(1)$ and~$e(-1)$ of~this edge. If~all vertex and~edge groups of~$\mathcal{G}(\Gamma)$ are~finite cyclic, then $\mathcal{G}(\Gamma)$ can be~replaced with~the~graph~$\mathcal{M}(\Gamma)$, in~which labels are~assigned not~only to~the~ends of~the~edges, but~also to~the~vertices. Namely, the~label~$\mu(v)$ at~a~vertex~$v$ means that the~group~$G_{v}$ is~of~order~$\mu(v)$. Of~course, the~equality $|\mu(e(1)) / \lambda(+e)| = |\mu(e(-1)) / \lambda(-e)|$ must hold for~each edge $e \in \mathcal{E}$. We~call the~group defined by~the~graph~$\mathcal{L}(\Gamma)$ ($\mathcal{M}(\Gamma)$) the~\emph{fundamental group of~the~graph with~labels~$\mathcal{L}(\Gamma)$}~($\mathcal{M}(\Gamma)$) and~denote it~by~$\pi_{1}(\mathcal{L}(\Gamma))$ (respectively,~$\pi_{1}(\mathcal{M}(\Gamma))$).

\section{Non-solv\-able GBS-groups. Proof of~Theorem~\ref{et01}}\label{es04}

It~follows from~the~last paragraph of~Section~\ref{es03} that each GBS-group can be~defined by~a~graph with~labels~$\mathcal{L}(\Gamma)$ for~some finite connected graph~$\Gamma$ and~vice versa, each graph with~labels~$\mathcal{L}(\Gamma)$ over~a~non-empty finite connected graph~$\Gamma$ defines some GBS-group. Throughout this section, we~assume that $\Gamma$ is~a~non-empty finite connected graph with~a~vertex set~$\mathcal{V}$ and~an~edge set~$\mathcal{E}$, $\mathcal{L}(\Gamma)$~is~a~graph with~labels~$\lambda(\varepsilon e)$, $e \in \mathcal{E}$,~$\varepsilon =\nolinebreak \pm 1$, and~$\mathfrak{G}$ is~the~corresponding GBS-group with~the~vertex groups $G_{v} = \langle g_{v} \rangle$, $v \in \mathcal{V}$, and the~edge subgroups $H_{\varepsilon e} = \big\langle g_{e(\varepsilon)}^{\lambda(\varepsilon e)} \big\rangle$, $e \in \mathcal{E}$, $\varepsilon = \pm 1$.

Suppose that $\mathcal{L}(\Gamma)$ contains an~edge~$e$ such that $e$ is~not~a~loop and~$|\lambda(\varepsilon e)| = 1$ for~some $\varepsilon = \pm 1$. If~we~choose a~maximal subtree of~$\Gamma$ containing~$e$ and~use~it to~define a~presentation of~the~group~$\mathfrak{G}$, then the~equality $g_{e(\varepsilon)}^{\vphantom{(}} = g_{e(-\varepsilon)}^{\lambda(\varepsilon e)\lambda(-\varepsilon e)}$ holds in~$\mathfrak{G}$ and,~therefore, the~generator~$g_{e(\varepsilon)}^{\vphantom{(}}$ can be~excluded from~this presentation. In~$\mathcal{L}(\Gamma)$, this operation corresponds to~the~contraction of~$e$ with~preliminary multiplication of~all labels around the~vertex~$e(\varepsilon)$ by~$\lambda(\varepsilon e)\lambda(-\varepsilon e)$. Such a~transformation of~$\mathcal{L}(\Gamma)$ is~called an~\emph{elementary collapse} (see~\cite[p.~480]{Levitt2007GT}).

The~graph~$\mathcal{L}(\Gamma)$ is~said to~be~\emph{reduced} if, for~any $e \in \mathcal{E}$, $\varepsilon = \pm 1$, it~follows from~the~equality $|\lambda(\varepsilon e)| = 1$ that $e$ is~a~loop~\cite[p.~224]{Forester2002GT}. Since $\Gamma$ is~finite, we~can always make $\mathcal{L}(\Gamma)$ reduced by~using a~finite number of~elementary collapses.

If~we~replace the~generator of~some vertex group with~its inverse, then all the~labels around the~corresponding vertex change sign. Similarly, replacing the~generator of~some edge group with~its inverse affects the~signs of~the~labels at~the~ends of~this edge. The~described replacements of~the~generators do~not change the~group~$\mathfrak{G}$, and~the~corresponding graph transformations are~called \emph{admissible changes of~signs}~\cite[p.~479]{Levitt2007GT}.

Let a~maximal subtree~$\mathcal{T}$ of~$\Gamma$ be~fixed. It~is~easy to~see that, by~applying suitable admissible changes of~signs, we~can make positive all the~labels at~the~ends of~the~edges of~$\mathcal{T}$. In~this case, we~say that the~graph~$\mathcal{L}(\Gamma)$ is~\emph{$\mathcal{T}$\nobreakdash-posi\-tive}.

An~element $a \in \mathfrak{G}$ is~called an~\emph{elliptic} one if it~is~conjugate to~an~element of~some vertex group. If~the~group~$\mathfrak{G}$ is~non-ele\-men\-tary, then any two elliptic elements $a,b \in \mathfrak{G} \setminus \{1\}$ are~commensurable, i.e.,~$\langle a \rangle \cap \langle b \rangle \ne 1$~\cite[Lemma~2.1]{Levitt2007GT}. This allows~us to~consider the~mapping $\Delta\colon \mathfrak{G} \to \mathbb{Q}^{*}$, which is~called the~\emph{modular homomorphism} of~$\mathfrak{G}$ and~is~defined as~follows.

Let $g \in \mathfrak{G}$. If~$a$ is~a~non-trivial elliptic element, then the~element $g^{-1}ag$ is~also elliptic and,~therefore, $g^{-1}a^{m}g = a^{n}$ for~some $m,n \in \mathbb{Z} \setminus \{0\}$. We~put $\Delta(g) = n / m$. It~is~known that the~number~$\Delta(g)$ does~not depend on~the~choice of~$a$, $m$, and~$n$~\cite{Kropholler1990CMH}.

The~largest cyclic normal subgroup of~$\mathfrak{G}$ is~called the~\emph{cyclic radical} of~this group and~is~denoted by~$C(\mathfrak{G})$. The~cyclic radical exists if $\mathfrak{G}$ is~not~isomorphic to~$\mathrm{BS}(1,1)$ or~$\mathrm{BS}(1,-1)$~\cite[p.~1808]{DelgadoRobinsonTimm2017CA}.

\begin{eproposition}\label{ep09}
Suppose that $\mathfrak{G}$ is~non-ele\-men\-tary and~$\operatorname{Im}\Delta \subseteq \{1,-1\}$. Suppose also that $\mathcal{T}$ is~a~maximal subtree of~$\Gamma$ used to~define a~presentation of~$\mathfrak{G}$\textup{,} $\mathcal{E}_{\mathcal{T}}$~is~the~edge set of~$\mathcal{T}$\textup{,} $\mathcal{L}(\Gamma)$~is~reduced and~$\mathcal{T}$\nobreakdash-posi\-tive. Then the~following statements hold.

\textup{1.\hspace{1ex}}The~relations
\[
1 \ne C(\mathfrak{G}) =\kern-4pt{} \bigcap_{\substack{e \in \mathcal{E},\\ \varepsilon = \pm 1}}\kern-3pt{} H_{\varepsilon e} \leqslant \bigcap_{v \in \mathcal{V}} G_{v}
\]
hold\textup{,} and~therefore the~number $\mu(v) = [G_{v} : C(\mathfrak{G})]$ is~defined and~finite for~each $v \in \mathcal{V}$.{\parfillskip=0pt{}\par}

\textup{2.\hspace{1ex}}The~least common multiple~$\mu$ of~the~numbers~$\mu(v)$\textup{,} $v \in \mathcal{V}$\textup{,} divides $\prod_{e \in \mathcal{E},\,\varepsilon = \pm 1} \lambda(\varepsilon e)$.{\parfillskip=0pt{}\par}

\textup{3.\hspace{1ex}}Suppose that $Z$ is~an~infinite cyclic group with~a~generator~$z$\textup{,} $A$~is~the~free abelian group with~the~basis $\{a_{q} \mid q \in \operatorname{Im}\Delta\}$\textup{,} and~$X$ is~the~splitting extension of~$Z$ by~$A$ such that $a_{q}^{-1}za_{q}^{\vphantom{1}} = z^{q}$\textup{,} $q \in \operatorname{Im}\Delta$. Then the~mapping of~the~generators of~$\mathfrak{G}$ into~$X$ given by~the~rule
\[
g_{v} \mapsto z^{\mu/\mu(v)}, \quad 
v \in \mathcal{V}, \quad
t_{e} \mapsto a_{\Delta(t_{e})}, \quad 
e \in \mathcal{E} \setminus \mathcal{E}_{\mathcal{T}},
\]
defines a~homomorphism $\tau \colon \mathfrak{G} \to X$.

\textup{4.\hspace{1ex}}Suppose that $k \geqslant 1$\textup{,} $Z_{k}^{\vphantom{\mu}} = \langle z_{k}^{\vphantom{\mu}};\ z_{k}^{\mu k} = 1 \rangle$\textup{,} $B = \langle b;\ b^{2} = 1 \rangle$\textup{,} $\eta_{k}\colon \mathfrak{G} \to \mathfrak{G} / (C(\mathfrak{G}))^{k}$ is~the~natural homomorphism\textup{,} and~$X_{k}$ is~either the~group~$Z_{k}^{\vphantom{1}}$ \textup{(}if~$\operatorname{Im}\Delta = \{1\}$\textup{)} or~the~splitting extension of~$Z_{k}^{\vphantom{1}}$ by~$B$ such that $b^{-1}z_{k}^{\vphantom{1}}b = z_{k}^{-1}$ \textup{(}if~$\operatorname{Im}\Delta = \{1,-1\}$\textup{)}. Then

\textup{a)\hspace{1ex}}the~map $\chi_{k}\colon X \to X_{k}$ given by~the~rule
\begin{align*}
a_{1}^{m}z^{\ell} &\mapsto z_{k}^{\ell},\phantom{b^{n}}\quad \ell,m \in \mathbb{Z},\phantom{,n}\quad 
\text{if}\ \operatorname{Im}\Delta = \{1\},\\
a_{1}^{m}a_{-1}^{n}z^{\ell} &\mapsto b^{n}z_{k}^{\ell},\quad \ell,m,n \in \mathbb{Z},\quad 
\text{if}\ \operatorname{Im}\Delta = \{1,-1\}
\end{align*}
is~a~homomorphism\textup{;}

\textup{b)\hspace{1ex}}the~map of~the~generators of~$\mathfrak{G} / (C(\mathfrak{G}))^{k}$ into~$X_{k}$ given by~the~rule
\[
g_{\vphantom{k}v}^{\vphantom{\mu}} \mapsto z_{k}^{\mu / \mu(v)}, \quad v \in \mathcal{V}, \quad 
t_{e} \mapsto \begin{cases}
1 & \text{if}\ \Delta(t_{e}) = 1, \\
b & \text{if}\ \Delta(t_{e}) = -1,
\end{cases} \quad 
e \in \mathcal{E} \setminus \mathcal{E}_{\mathcal{T}},
\]
defines a~homomorphism $\tau_{k}\colon \mathfrak{G} / (C(\mathfrak{G}))^{k} \to X_{k}$\textup{;}

\textup{c\kern1pt)\hspace{1ex}}the~equality $\tau\chi_{k} = \eta_{k}\tau_{k}$ holds\textup{;}

\textup{d)\hspace{1ex}}the~kernel of~$\tau_{k}$ is~a~free group.
\end{eproposition}

\begin{proof}
Statements~1\nobreakdash--3 are~special cases of~Propositions 4.4 and~5.1 from~\cite{Sokolov2021JA}. Let~us prove Statement~4.

The~fact that $\chi_{k}$ is~a~homomorphism can be~verified directly. Let $\xi_{k}$ be~the~mapping of~words that continues the~map from~Statement~4\nobreakdash-b. It~is~obvious that $\tau\chi_{k}(g) = \eta_{k}\xi_{k}(g)$ for~each generator~$g$ of~$\mathfrak{G}$. Therefore, the~mapping~$\xi_{k}$ takes all defining relations of~$\mathfrak{G} / (C(\mathfrak{G}))^{k}$ inherited from~$\mathfrak{G}$ into~the~equalities valid in~$X_{k}$. In~addition to~them, the~presentation of~the~quotient group $\mathfrak{G} / (C(\mathfrak{G}))^{k}$ contains relations of~the~form $\omega = 1$, where $\omega$ is~an~arbitrary word that defines an~element of~$(C(\mathfrak{G}))^{k}$. Let $u \in \mathcal{V}$. Since $\mu(u) = [G_{u} : C(\mathfrak{G})]$, the~subgroup~$(C(\mathfrak{G}))^{k}$ is~generated by~the~element~$g_{u}^{\mu(u)k}$. Therefore, all these additional relations are~derivable from~the~relation $g_{u}^{\mu(u)k} = 1$. It~follows from~the~definitions of~$\xi_{k}$ and~$Z_{k}$ that $\xi_{k}$ takes the~latter relation into~the~equality valid in~$X_{k}$. Thus, $\tau_{k}$~is~a~homomorphism and~$\tau\chi_{k} = \eta_{k}\tau_{k}$.

It~is~obvious that the~quotient group $\mathfrak{G} / (C(\mathfrak{G}))^{k}$ is~the~fundamental group of~the~graph with~labels~$\mathcal{M}(\Gamma)$ which is~obtained from~$\mathcal{L}(\Gamma)$ by~assigning to~each vertex $v \in \mathcal{V}$ the~label~$\mu(v)k$. Since the~order of~the~group $Z_{k}^{\mu/\mu(v)} = \langle g_{\vphantom{k}v}^{\vphantom{\mu}} \rangle\tau_{k}^{\vphantom{\mu}}$ is~also equal to~$\mu(v)k$, the~homomorphism~$\tau_{k}^{\vphantom{\mu}}$ acts injectively on~all vertex groups. Now it~follows from~Proposition~\ref{ep08} that its kernel is~a~free group.
\end{proof}

\begin{eproposition}\label{ep10}
\textup{\cite[Theorem~3]{Sokolov2021JA}}
Suppose that $\mathcal{C}$ is~a~root class of~groups consisting only of~periodic groups\textup{,} the~group~$\mathfrak{G}$ is~non-solv\-able\textup{,} and~the~graph~$\mathcal{L}(\Gamma)$ is~reduced. Then the~following statements hold.

\textup{1.\hspace{1ex}}If~$\operatorname{Im}\Delta = \{1\}$\textup{,} then $\mathfrak{G}$ is~residually a~$\mathcal{C}$\nobreakdash-group if and~only~if all the~labels of~$\mathcal{L}(\Gamma)$ are~$\mathfrak{P}(\mathcal{C})$\nobreakdash-num\-bers.

\textup{2.\hspace{1ex}}If~$\operatorname{Im}\Delta = \{1,-1\}$\textup{,} then $\mathfrak{G}$ is~residually a~$\mathcal{C}$\nobreakdash-group if and~only~if all the~labels of~$\mathcal{L}(\Gamma)$ are~$\mathfrak{P}(\mathcal{C})$\nobreakdash-num\-bers and~$2 \in \mathfrak{P}(\mathcal{C})$.

\textup{3.\hspace{1ex}}If~$\operatorname{Im}\Delta \nsubseteq \{1,-1\}$\textup{,} then $\mathfrak{G}$ is~not~residually a~$\mathcal{C}$\nobreakdash-group.
\end{eproposition}

\begin{eproposition}\label{ep11}
\textup{\cite[Theorem~1]{Sokolov2015MN}}
If~$\mathcal{C}$ is~a~root class of~groups consisting only of~finite groups\textup{,} then any extension of~a~free group by~a~$\mathcal{C}$\nobreakdash-group is~conjugacy $\mathcal{C}$\nobreakdash-sepa\-ra\-ble.
\end{eproposition}

\begin{eproposition}\label{ep12}
Suppose that the~group~$\mathfrak{G}$ is~non-solv\-able\textup{,} $\mathcal{T}$~is~a~maximal subtree of~$\Gamma$ used to~define a~presentation of~$\mathfrak{G}$\textup{,} the~graph~$\mathcal{L}(\Gamma)$ is~reduced and~$\mathcal{T}$\nobreakdash-posi\-tive. If~$\mathcal{C}$ is~a~root class of~groups consisting only of~periodic groups and~$\mathfrak{G}$ is~residually a~$\mathcal{C}$\nobreakdash-group\textup{,} then $\mathfrak{G}$ is~conjugacy $\mathcal{FS}_{\mathfrak{P}(\mathcal{C})}$\nobreakdash-sepa\-ra\-ble.
\end{eproposition}

\begin{proof}
It~follows from~Proposition~\ref{ep10} that $\operatorname{Im}\Delta \subseteq \{1,-1\}$, all the~labels of~$\mathcal{L}(\Gamma)$ are $\mathfrak{P}(\mathcal{C})$\nobreakdash-num\-bers, and~if $\operatorname{Im}\Delta = \{1,-1\}$, then $2 \in \mathfrak{P}(\mathcal{C})$. Given $k \geqslant 1$, let $Z$, $Z_{k}$, $A$, $B$, $X$, $X_{k}$ and~$\tau$, $\tau_{k}$, $\chi_{k}$, $\eta_{k}$ be~the~groups and~the~homomorphisms from~Statements~3 and~4 of~Proposition~\ref{ep09}. By~Statement~2 of~the~same proposition, $\mu$~is~a~$\mathfrak{P}(\mathcal{C})$\nobreakdash-num\-ber. Therefore, $X_{k} \in \mathcal{FS}_{\mathfrak{P}(\mathcal{C})}$ if $k$ is~a~$\mathfrak{P}(\mathcal{C})$\nobreakdash-num\-ber.

Suppose that $f,g \in \mathfrak{G}$ and~$f \nsim_{\mathfrak{G}} g$. Since the~class~$\mathcal{FS}_{\mathfrak{P}(\mathcal{C})}$ is~closed under~taking subgroups, to~complete the~proof, it~suffices to~find a~homomorphism~$\sigma$ of~$\mathfrak{G}$ into~an~$\mathcal{FS}_{\mathfrak{P}(\mathcal{C})}$\nobreakdash-group~$\overline{\mathfrak{G}}$ such that $f\sigma \nsim_{\overline{\mathfrak{G}}} g\sigma$. Let us consider two cases.

\smallskip

\textit{Case~1.} $f \nsim_{\mathfrak{G}} g \pmod {(C(\mathfrak{G}))^{k}}$ for~some $\mathfrak{P}(\mathcal{C})$\nobreakdash-num\-ber $k \geqslant 1$.

\smallskip

By~Statement~4\nobreakdash-d of~Proposition~\ref{ep09}, the~quotient group $\mathfrak{G} / (C(\mathfrak{G}))^{k}$ is~an~extension of~a~free group by~the~group $(\mathfrak{G} / (C(\mathfrak{G}))^{k})\tau_{k}$. Since $(\mathfrak{G} / (C(\mathfrak{G}))^{k})\tau_{k} \leqslant X_{k} \in \mathcal{FS}_{\mathfrak{P}(\mathcal{C})}$ and~$\mathcal{FS}_{\mathfrak{P}(\mathcal{C})}$ is~a~root class, the~group $\mathfrak{G} / (C(\mathfrak{G}))^{k}$ is~conjugacy $\mathcal{FS}_{\mathfrak{P}(\mathcal{C})}$\nobreakdash-sepa\-ra\-ble due~to Proposition~\ref{ep11}. It~follows from~this fact and~the~relation $f\eta_{k} \nsim_{\mathfrak{G}\eta_{k}} g\eta_{k}$ that $\eta_{k}$ can be~continued to~the~desired homomorphism.

\smallskip

\textit{Case~2.} $f \sim_{\mathfrak{G}} g \pmod {(C(\mathfrak{G}))^{k}}$ for~any $\mathfrak{P}(\mathcal{C})$\nobreakdash-num\-ber $k \geqslant 1$.

\smallskip

Since $X$ is~a~splitting extension of~$Z$ by~$A$, the~element~$g\tau$ can be~written in~the~form $g\tau = az^{n}$ for~suitable $a \in A$ and~$n \in \mathbb{Z}$. Let $m$ be~a~$\mathfrak{P}(\mathcal{C})$\nobreakdash-num\-ber greater than~$2|n|$. By~the~supposition, $f \sim_{\mathfrak{G}} g \pmod {(C(\mathfrak{G}))^{m}}$ and,~therefore, $x^{-1}fx = gh$ for~some $x \in \mathfrak{G}$ and~$h \in (C(\mathfrak{G}))^{m} \setminus \{1\}$. Hence, we~can assume further without loss of~generality that $f = gh$. By~the~definition of~$\tau$, the~equality $h\tau = z^{m\ell}$ holds for~some $\ell \ne 0$. Let~us choose a~$\mathfrak{P}(\mathcal{C})$\nobreakdash-num\-ber~$k$ greater than~$2m|\ell|$ and~show that $g\tau\chi_{k} \ne (gh)\tau\chi_{k} \ne g^{-1}\tau\chi_{k}$.

Indeed,
\[
(g^{2}h)\tau = (az^{n})(az^{n})z^{m\ell} = a^{2}z^{n + \varepsilon n + m\ell},
\]
where $\varepsilon = 1$ if $a$ belongs to~the~centralizer of~$z$ in~$X$ and~$\varepsilon = -1$ otherwise. By~the~definition of~$\chi_{k}^{\vphantom{m}}$, the~equalities $h\tau\chi_{k}^{\vphantom{m}} = z_{k}^{m\ell}$ and~$(g^{2}h)\tau\chi_{k}^{\vphantom{m}} = z_{k}^{n + \varepsilon n + m\ell}$ hold. The~relations $m > 2|n|$, $k > 2m|\ell|$, and~$\ell \ne 0$ imply that $0 \ne |m\ell| < k/2 < \mu k$~and
\begin{multline*}
0 < m - 2|n| \leqslant m|\ell| - 2|n| \leqslant |n + \varepsilon n + m\ell| \leqslant\\
\leqslant m|\ell| + 2|n| < m|\ell| + m \leqslant 2m|\ell| < k \leqslant \mu k.
\end{multline*}
Since the~group~$Z_{k}$ is~of~order~$\mu k$, it~follows that $h\tau\chi_{k} \ne 1 \ne (g^{2}h)\tau\chi_{k}$ and,~therefore, $g\tau\chi_{k} \ne (gh)\tau\chi_{k} \ne g^{-1}\tau\chi_{k}$ as~required.

Let~us now show that the~conjugacy class~$(g\tau\chi_{k})^{X_{k}}$ of~$g\tau\chi_{k}$ in~$X_{k}$ is~contained in~the~set $\{g\tau\chi_{k}, g^{-1}\tau\chi_{k}\}$. Since $X_{k} \in \mathcal{FS}_{\mathfrak{P}(\mathcal{C})}$, it~will follow that $\tau\chi_{k}$ is~the~desired homomorphism.

If~$\operatorname{Im}\Delta = \{1\}$, then $X_{k}$ is~an~abelian group and,~therefore, $(g\tau\chi_{k})^{X_{k}} = \{g\tau\chi_{k}\}$. Suppose that $\operatorname{Im}\Delta = \{1,-1\}$. Let~us fix a~vertex $v \in \mathcal{V}$, denote by~$c$ the~element~$g_{v}^{\mu(v)}$ and~assume that $[g,c] \ne 1$.

Since $c$ generates the~infinite cyclic subgroup~$C(\mathfrak{G})$, which is~normal in~$\mathfrak{G}$, the~equality $g^{-1}cg = c^{-1}$ holds. It~follows that $g^{-1}c^{-1}g = c$, $c^{-1}gc = gc^{2}$, $cgc^{-1} = gc^{-2}$, and~therefore $c^{-r}gc^{r} = gc^{2r}$ for~any $r \in \mathbb{Z}$. The~relation $\operatorname{Im}\Delta = \{1,-1\}$ implies that $2 \in \mathfrak{P}(\mathcal{C})$. Hence, $f \sim_{\mathfrak{G}} g \pmod {(C(\mathfrak{G}))^{2}}$ and~$y^{-1}fy = gc^{2s}$ for~some $y \in \mathfrak{G}$ and~$s \in \mathbb{Z}$. Since $gc^{2s} = c^{-s}gc^{s}$, it~follows that the~elements $f$ and~$g$ are~conjugate in~$\mathfrak{G}$, and~we~get a~contradiction with~their choice.

Thus, $[g,c] = 1$ and,~therefore, the~element~$g\tau\chi_{k}$ belongs to~the~centralizer~$\mathcal{Z}_{X_{k}}(c\tau\chi_{k})$ of~$c\tau\chi_{k}$ in~$X_{k}$. Let~us note that, if $u \in Z_{k}$, then $w^{-1}uw = u$ and~$(bw)^{-1}u(bw) = u^{-1}$ for~any $w \in Z_{k}$. Hence, the~conjugacy class of~$u$ in~$X_{k}$ is~the~set $\{u, u^{-1}\}$, while the~centralizer of~$u$ is~either~$Z_{k}$ (if~$u^{2} \ne 1$) or~$X_{k}$ (otherwise). By~the~definitions of~$\tau$ and~$\chi_{k}$, the~equality $c\tau\chi_{k}^{\vphantom{\mu}} = z_{k}^{\mu}$ holds. Since the~group~$Z_{k}$ is~of~order~$\mu k$ and~$k > 2m|\ell| > 2$, it~follows that $(c\tau\chi_{k})^{2} \ne 1$. Hence, $g\tau\chi_{k} \in \mathcal{Z}_{X_{k}}(c\tau\chi_{k}) = Z_{k}$ and,~therefore, $(g\tau\chi_{k})^{X_{k}} = \{g\tau\chi_{k}, g^{-1}\tau\chi_{k}\}$ as~required.
\end{proof}

\begin{proof}[\textup{\textbf{Proof of~Theorem~\ref{et01}}}]
The~implication~$2 \Rightarrow 1$ is~obvious, the~implication~$3 \Rightarrow 2$ follows from~Proposition~\ref{ep04}. Let~us prove the~implication~$1 \Rightarrow 3$.

By~using elementary collapses and~admissible changes of~signs, we~can transform $\mathcal{L}(\Gamma)$ at~first to~a~reduced form, and~then, after choosing some maximal subtree~$\mathcal{T}$ of~$\Gamma$, to~a~$\mathcal{T}$\nobreakdash-posi\-tive form. Therefore, if~$\mathfrak{G}$ is~non-solv\-able, then the~considered implication follows from~Proposition~\ref{ep12}. If~$\mathfrak{G} \cong \mathrm{BS}(1,-1)$, the~implication~$1 \Rightarrow 3$ is~ensured by~Proposition~\ref{ep06}.

Let $\mathfrak{G} \cong \mathbb{Z}$ or~$\mathfrak{G} \cong \mathrm{BS}(1,1) \cong \mathbb{Z}\times \mathbb{Z}$. Then $\mathfrak{G}$ is~residually a~$\mathcal{C}$\nobreakdash-group if and~only~if it~is~conjugacy $\mathcal{C}$\nobreakdash-sepa\-ra\-ble because the~relations of~equality and~conjugacy coincide for~an~abelian group. Since the~image of~any homomorphism of~$\mathfrak{G}$ onto~a~group from~$\mathcal{C}$ is, obviously, an~$\mathcal{FS}_{\mathfrak{P}(\mathcal{C})}$\nobreakdash-group, it~follows that the~implication~$1 \Rightarrow 3$ also holds.
\end{proof}

\end{document}